\pdfoutput=1

\documentclass[12pt,
headings=small,
final
]{scrartcl}

 \usepackage[a4paper, left=2.7cm, right=2.7cm]{geometry}    %
\setkomafont{sectioning}{\rmfamily\bfseries} %

\usepackage{pifont}
\usepackage{times}
\usepackage{amsfonts}       %
\usepackage{amsmath} %
\usepackage{amssymb}   
\usepackage{amsthm}     
\usepackage{booktabs}
\usepackage{latexsym}
 \usepackage[matrix,curve]{xypic}
\usepackage{microtype}        
\usepackage{titlefoot}
\usepackage{pdflscape}
\usepackage{longtable, tabularx, ltxtable}
\usepackage{filecontents}
\newcommand{\R}{\mathbb{R}}
\newcommand{\C}{\mathbb{C}}

\newcommand{\CCC}{\mathcal{C}}

\newtheorem{thm}{Theorem}[section]
\newtheorem{cor}[thm]{Corollary}
\newtheorem{lem}[thm]{Lemma}
\newtheorem{prop}[thm]{Proposition}

\theoremstyle{definition}

\newcommand{\diag}{\mathrm{diag}}
\newcommand{\PS}{\mathrm{PS}}
\newcommand{\ex}{\mathrm{ex}}
\newcommand{\reg}{\mathrm{reg}}
\newcommand{\St}{\mathrm{St}}

\newcommand{\GSp}{\mathrm{GSp}}
\newcommand{\Gl}{\mathrm{Gl}}
\newcommand\nosf[1]{\begin{footnotesize}\textup{\textsf{#1}}\end{footnotesize}}

\DeclareMathOperator{\Hom}{Hom}
\DeclareMathOperator{\Ext}{Ext}
\DeclareMathOperator{\Diff}{d}

\DeclareMathOperator{\ind}{ind}

\usepackage{enumitem}
\setlist[enumerate]{itemsep=-0.5ex plus0.1ex minus 0.2ex}
\setlist[description]{itemsep=-0.5ex plus0.1ex minus 0.2ex}
\setlist[itemize]{itemsep=-0.5ex plus0.1ex minus 0.2ex}
\let\svthefootnote\thefootnote
\newcommand\blankfootnote[1]{%
  \let\thefootnote\relax\footnotetext{#1}%
  \let\thefootnote\svthefootnote%
}

\everydisplay{
\thickmuskip=5mu plus 2mu
\medmuskip=4mu plus 1mu
\relax} 
 
\renewcommand{\Re}{\mathrm{Re}}

\title{\begin{large}Exceptional poles of local spinor $L$-functions of $\mathrm{GSp}(4)$\\ with anisotropic Bessel models\end{large}}
\date{}
\author{Mirko R\"osner\and Rainer Weissauer}
\begin{document}\maketitle

 \begin{abstract}
  For non-cuspidal irreducible admissible representations of $\GSp(4,k)$ over a local non-archimedean field $k$, we determine the exceptional poles of the spinor $L$-factor attached to anisotropic Bessel models by Piatetski-Shapiro.\blankfootnote{2010 \textit{Mathematics Subject Classification.} Primary 22E50; Secondary 11F46, 11F70, 20G05.}
 \end{abstract}

\section{Introduction}
Fix a non-archimedean local field $k$ of characteristic not two. %
For infinite-dimensional irreducible admissible representations $\Pi$ of the group $\GSp(4,k)$ of symplectic similitudes of genus two, a spinor $L$-factor $L^{\PS}(s,\Pi,\Lambda,\mu)$ has been constructed by Piatetski-Shapiro \cite{PS-L-Factor_GSp4}, attached to a choice of a Bessel model $(\Lambda,\psi)$ and a smooth character $\mu$ of $k^\times$.
Bessel models have been classified by Roberts and Schmidt \cite{Roberts-Schmidt_Bessel}.
The $L$-factor factorizes into a regular part and an exceptional part
\begin{equation*}
 L^{\PS}(s,\Pi,\Lambda,\mu)=L^{\PS}_{\reg}(s,\Pi,\Lambda,\mu)L^{\PS}_{\ex}(s,\Pi,\Lambda,\mu)\ .
\end{equation*}
The regular factors were determined by  Dani\c{s}man \cite{Danisman, Danisman2, Danisman3} and by the authors \cite{RW, Subregular}. Concerning the exceptional factor %
Piatetski-Shapiro \cite[thm.~4.2]{PS-L-Factor_GSp4} has shown that if $L^{\PS}_{\ex}(s,\Pi,\Lambda,\mu)$ is non-trivial, then $\mu\otimes\Pi$ admits a non-zero functional, equivariant with respect to an unramified character of a certain subgroup.
He expects that the converse assertion holds as well.

For the case of anisotropic Bessel models and non-cuspidal $\Pi$ we verify this expectation in theorem~\ref{thm:exceptional_poles}.
Furthermore, we classify the non-cuspidal $\Pi$ that admit such functionals, see theorem~\ref{thm:classification_H-functionals}.
In other words, for non-cuspidal irreducible admissible representations of $\GSp(4,k)$ with anisotropic Bessel model we calculate the exceptional factor $L^{\PS}_{\ex}(s,\Pi,\Lambda,\mu)$ as a product of Tate local $L$-factors, see corollary~\ref{cor:final}.

The exceptional factors are given in table~\ref{tab:spinor-L-factor} in the cases where they are non-trivial.
For the irreducible admissible representations $\Pi$ we use the classification symbols of \cite{Sally-Tadic}
and \cite{Roberts-Schmidt}.
The column $\Lambda$ lists the smooth characters that give rise to an anisotropic Bessel model $(\Lambda,\psi)$ under the indicated condition. Here $\chi_{K/k}$ denotes the quadratic character attached to the field extension $K/k$, underlying the Bessel model.
The right column lists the exceptional part of the spinor $L$-factor as a product of Tate $L$-factors.
For each non-cuspidal $\Pi$ that is not displayed in the table, there is either no anisotropic Bessel model or the exceptional factor is trivial.

The analogous question for cuspidal $\Pi$ has been answered by Dani\c{s}man \cite{Danisman_Annals} and for split Bessel models by Weissauer \cite{W_Excep}.
To tie these results together, we define the class of \emph{extended Saito-Kurokawa representations}.
These are the non-generic non-cuspidal $\Pi$ of type \nosf{IIb}, \nosf{Vbcd}, \nosf{VIbcd}, \nosf{XIb} and the non-generic cuspidal $\Pi$ of type \nosf{Va*} and \nosf{XIa*} in the sense of \cite{Roberts-Schmidt_Bessel}.
Only irreducible representations in this class  admit a non-trivial exceptional $L$-factor.
For the convenience of the reader,  these results will be reviewed in section~\ref{s:tables}. 

\begin{table}
\begin{center}
\caption{Exceptional part of Piateskii-Shapiro $L$-factors for anisotropic Bessel models.\label{tab:spinor-L-factor}}
\begin{small}
\begin{tabular}{lllll}
\toprule
Type       & $\Pi\in\CCC_G$         & Condition &  $\Lambda$                 & $L_{\ex}^{\PS}(s,\Pi,\Lambda,\mu)$ \\
\midrule
\nosf{IIb} & $(\chi\circ\det)\rtimes\sigma$ &   & $\chi\sigma\circ N_{K/k}$  & $L(s,\nu^{1/2}\mu\chi\sigma)$  \\
\nosf{Vb}  & $L(\nu^{1/2}\xi \St,\nu^{-1/2}\sigma)$ & $\xi\neq\chi_{K/k}$    & $\sigma\circ N_{K/k}$&$L(s,\nu^{1/2}\mu\sigma)$\\
\nosf{Vc}  & $L(\nu^{1/2}\xi \St,\xi\nu^{-1/2}\sigma)$ & $\xi\neq\chi_{K/k}$ & $\xi\sigma\circ N_{K/k}$&$L(s,\nu^{1/2}\mu\xi\sigma)$\\
\nosf{Vd}  & $L(\nu\xi,\xi\rtimes\nu^{-1/2}\sigma)$ & $\xi=\chi_{K/k}$       & $\sigma\circ N_{K/k}$&$L(s,\nu^{1/2}\mu\sigma)L(s,\nu^{1/2}\mu\xi\sigma)$\\
\nosf{VIb} & $\tau(T,\nu^{-1/2}\sigma)$         &                            & $\sigma\circ N_{K/k}$&$L(s,\nu^{1/2}\mu\sigma)$\\
\nosf{XIb} & $L(\nu^{1/2}\pi,\nu^{-1/2}\sigma)$ & $\pi_{\widetilde{T}}\neq0$ & $\sigma\circ N_{K/k}$&$L(s,\nu^{1/2}\mu\sigma)$\\
\bottomrule
\end{tabular}
\end{small}
\end{center}
\end{table}

\section{Preliminaries}
For the non-archimedean local field $k$ of characteristic not two denote the integers by $\mathcal{O}_k$ and fix a uniformising element $\varpi$. Let $q$ be the cardinality of the residue field $\mathcal{O}_k/\varpi\mathcal{O}_k$.
The valuation character of $k^\times$ is $\nu(x)=\vert x\vert$.
The group of symplectic similitudes over $k$ in four variables (Siegel's notation) is
\begin{equation*}
G = \GSp(4,k)=\{ g\in \Gl(4,k) \mid g^t J g = \lambda(g) \cdot  J\} \ , \qquad J=\begin{pmatrix} 0 & E_2 \\ - E_2 & 0 \end{pmatrix}
\end{equation*}
with symplectic similitude factor $\lambda(g)\in k^\times$ and center $Z_G\cong k^\times$.
Fix a standard symplectic basis $e_1,e_2,e_1^*,e_2^*$ with polarisation $k^4=L_0\oplus L_0^*$ where $L_0=\C e_1\oplus \C e_2$.
Let $P=M\ltimes N$ be the Siegel parabolic subgroup that preserves $L_0$ and let $Q$ be the standard Klingen parabolic that preserves $\C e_1$.
The Weyl group of $G$ has order eight and is generated by \begin{equation*}
\mathbf{s}_1=\begin{pmatrix}0&1&0&0\\1&0&0&0\\0&0&0&1\\0&0&1&0\end{pmatrix}\quad,\qquad
\mathbf{s}_2=\begin{pmatrix}1&0&0&0\\0&0&0&1\\0&0&1&0\\0&-1&0&0\end{pmatrix}\ .
\end{equation*}

Fix a quadratic field extension $K=k(\sqrt{\alpha})$ with a non-square $\alpha\in k^\times$ and integers $\mathcal{O}_K$.
Every $a\in K$ is of the form $a=a_1+a_2\sqrt{\alpha}$ with unique coefficients $a_1,a_2\in k$.
The Galois group acts by the involution sending $a$ to $\overline{a}=a_1-a_2\sqrt{\alpha}$.
The symplectic form on $V=K^2$, $$\left<\left(\begin{smallmatrix}v\\w\end{smallmatrix}\right),\left(\begin{smallmatrix}x\\y\end{smallmatrix}\right)\right>=\tfrac12 \mathrm{tr}_{K/k}(vy-wx)\ ,\qquad v,w,x,y\in K\ ,$$ is preserved  by the natural action of %
$H=\{g\in \Gl(2,K)\mid \det g\in k^\times\}$ on $V$ up to the similitude factor $\det g\in k^\times$.
The symplectic basis $e_1=(1,0)^t$, $e_2=(\sqrt{\alpha},0)^t$, $e_1^*=(0,1)^t$ and $e_2^*=(0,1/\sqrt{\alpha})^t$ in $V$ gives rise to the embedding\footnote{The symplectic form in \cite[\S1(b)]{PSS81} omits the factor $1/2$. For the same $e_1$, $e_2$, this leads to $e_1^*=(0,1/2)^t$ and $e_2^*=(0,1/(2\sqrt{\alpha}))^t$ and thus to correction factors in the embedding \cite[(1.2)]{PSS81}.}
\begin{equation*}H\ni
 \begin{pmatrix}a&b\\c&d\end{pmatrix}\longmapsto
 \begin{pmatrix}
 a_1     & a_2\alpha & b_1     & b_2     \\
 a_2     & a_1     & b_2     & b_1/\alpha\\
 c_1     & c_2\alpha & d_1     & d_2     \\
 c_2\alpha & c_1\alpha & d_2\alpha & d_1     
 \end{pmatrix}\in G\ .
\end{equation*}
We identify $H$ with its image in $G$. Fix the tori $T=\{x_\lambda=\diag(\lambda,\lambda,1,1)\in G\mid\lambda\in k^\times\}$ and
$$\widetilde{T}=\left\{ \widetilde{t}_a=
\diag(a,\overline{a})=
\begin{pmatrix}
a_1 & a_2\alpha&0       &0   \\
a_2 & a_1    &0       &0   \\
0   & 0      &a_1     &-a_2\\
0   & 0      &-a_2\alpha&a_1
\end{pmatrix}\in G
\mid a\in K^\times \right\}\ .$$
The similitude factor of $\widetilde{t}_a$ is $\lambda_G(\widetilde{t}_a)=a_1^2-a_2^2\alpha=N_{K/k}(a)\in k^\times$. 
\begin{equation*}
\widetilde{N}=H\cap N=\left\{
\begin{pmatrix}
1 & 0 & b_1 & b_2\\
0 & 1 & b_2 & b_1/\alpha\\
0 & 0 & 1   & 0\\
0 & 0 & 0   & 1
\end{pmatrix}\mid b\in K\right\}
\end{equation*}
is the unipotent radical of the standard Borel $B_H=T\widetilde{T}\widetilde{N}$ of $H$. Its complement $S$ in $N$ is the centralizer of $\widetilde{T}N$
\begin{equation*}
 S=C_G(\widetilde{T}N)=\left\{
 \begin{pmatrix}
 1 & 0 & c& 0\\
 0 & 1 & 0&-c/\alpha\\
 0 & 0 & 1& 0\\
 0 & 0 & 0& 1
 \end{pmatrix}
 \mid c\in k\right\}\quad, \qquad N\cong \widetilde{N}\oplus S\ .
\end{equation*}

\subsection{Smooth representations}

For a locally compact totally disconnected group $X$ the modulus character $\delta_X:X\to\R_{>0}$ is defined by
$$\int_X f(xx_0) d x = \delta_X(x_0) \int_X f(x) d x\ ,\qquad x,x_0\in X\ ,\quad f\in C^\infty_c(X)$$ with respect to a left-invariant Haar measure on $X$.
We denote by $\CCC_X$ the abelian category of smooth complex-valued linear representations of $X$. %
For a closed subgroup of $Y\subseteq X$ with a smooth representation $\sigma$ of $Y$,
the \emph{unnormalized compact induction} of $\sigma$ is the representation $\ind_Y^X(\sigma)\in\CCC_X$,
given by the right-regular action of $X$ on the vector space of functions
$f:X\to \C$ that satisfy $f(yxk) = \sigma(y)f(x)$ for every $y\in Y, x\in X$ and $k$ in some open compact subgroup of $X$ and such that $f$ has compact support modulo $Y$.
This defines an exact functor $\ind_Y^X$ from $\CCC_Y$ to $\CCC_X$.
\begin{lem}[Frobenius reciprocity]\label{lem:Frobenius}
For finite-dimensional smooth representations $\rho$ of $X$ and arbitrary smooth representations $\sigma$ of $Y$, there is a natural isomorphism
$$\Hom_X(\ind_Y^X(\sigma),\rho) \cong \Hom_Y(\sigma,\rho\otimes \delta_Y\delta_X^{-1})\ .$$
\end{lem}
\begin{proof}
See \cite[2.29]{Bernstein-Zelevinsky76}. The algebraic dual of a finite-dimensional smooth representation is always smooth.
\end{proof}

\subsection{Bessel models and zeta integrals}
An \emph{anisotropic Bessel datum} in the standard form of \cite{PSS81} is a pair of smooth characters $\Lambda$ of $\widetilde{T}$ and $\psi$ of $N=\widetilde{N}\oplus S$, where $\psi$ is trivial on $\widetilde{N}$ and non-trivial on $S$.
Such a Bessel datum is said to provide an \emph{anisotropic Bessel model} $(\Lambda,\psi)$ to an irreducible admissible representation $\Pi$ of $G$, if there is a non-zero $\widetilde{T}N$-equivariant Bessel functional $\ell\in\Hom_{\widetilde{T}N}(\Pi,\Lambda\boxtimes\psi)$. At least for non-cuspidal $\Pi$, such an $\ell$ is unique up to scalars %
and a classification of the Bessel models has been given by Roberts and Schmidt \cite[thms.~6.2.2, 6.3.2]{Roberts-Schmidt_Bessel}.
The \emph{Bessel function} attached to $v\in \Pi$ is $W_v(g)=\ell(\Pi(g)v)$ for $g\in G$.

Fix an anisotropic Bessel model $(\Lambda,\psi)$ of $\Pi$ and a smooth character $\mu$ of $k^\times$. Attached to $v\in\Pi$ and a Schwartz-Bruhat function $\Phi\in C_c^\infty(K^2)$  is the zeta-function
\begin{align*}
 Z^{\PS}(s,v,\Lambda,\Phi,\mu)=\int_{\widetilde{N}\backslash H}W_v(h)\Phi((0,1)h)\mu(\lambda_G(h))
 |\lambda_G( h)|^{s+\tfrac12}\Diff h\ .
\end{align*}
The integral converges for sufficiently large $\Re(s)>0$ and admits a unique meromorphic continuation to the complex plane.
The spinor $L$-factor $L^{\PS}(s,\Pi,\Lambda,\mu)$ of Piatetski-Shapiro and Soudry \cite{PS-L-Factor_GSp4, PSS81} is defined as the regularization $L$-factor of these zeta-functions varying over $\Phi$ and $v$.

The \emph{regular factor} $L_{\reg}^{\PS}(s,\Pi,\Lambda,\mu)$ is the regularization $L$-factor of the family of zeta-functions $Z^{\PS}(s,v,\Lambda,\Phi,\mu)$ that are subject to the condition $\Phi(0,0)=0$.
It coincides  with the regularization $L$-factor of the \emph{regular zeta-functions}
\begin{equation*}Z_{\reg}^{\PS}(s,W_v,\mu)=\int_{T} W_v(x_\lambda) \mu(\lambda)|\lambda|^{s-\tfrac32} \Diff x_\lambda\ ,\qquad v\in\Pi\,,\end{equation*}
again convergent for sufficiently large $\Re(s)>0$ with unique meromorphic continuation, see \cite[thm.~4.1]{PS-L-Factor_GSp4}, \cite[prop.~2.5]{Danisman}.
The \emph{exceptional factor} $L_{\ex}^{\PS}(s,\Pi,\Lambda,\mu)$ is the regularization $L$-factor of the family of meromorphic functions
\begin{equation*}
 Z^{\PS}(s,v,\Lambda,\Phi,\mu)/L^{\PS}_{\reg}(s,\Pi,\Lambda,\mu)\ ,
\end{equation*}
varying over $v$ and $\Phi$. It suffices to consider $\Phi=\Phi_0$, the characteristic function of $\mathcal{O}_K\times\mathcal{O}_K$.
Attached to each smooth character $\Lambda$ of $\widetilde{T}$ is the \emph{Bessel functor} $\beta_\Lambda$ that sends $\Pi\in\CCC_G$ to the coinvariant quotient $\widetilde{\Pi}=\Pi_{\widetilde{T}\widetilde{N},\Lambda}\in\CCC_{TS}$ on which $\widetilde{T}\widetilde{N}$ acts by $\Lambda$.
The action of $TS$ on $\Pi$ carries over to $\widetilde{\Pi}$ because $\widetilde{T}\widetilde{N}$ is normalized by $TS$.
 Since $\widetilde{T}\widetilde{N}$ is compactly generated modulo center, the Bessel functor is exact.
We use the same notation for the functor sending an $H$-module $\pi\in\CCC_H$ to the coinvariant quotient $\widetilde{\pi}\in\CCC_T$ on which $\widetilde{T}\widetilde{N}$ acts by $\Lambda$.
Whenever $(\Lambda,\psi)$ provides a Bessel model for an irreducible $\Pi\in\CCC_G$, the Bessel functional $\ell$ factorizes over $\widetilde{\Pi}$.
\begin{lem}\label{Bessel_module_perfect}
For irreducible admissible $\Pi\in\CCC_G$ and every smooth character $\Lambda$ of $K^\times$, the Bessel module $\widetilde{\Pi}$ has finite length. If $(\Lambda,\psi)$ yields an anisotropic Bessel model for $\Pi$, then $\widetilde{\Pi}$ is perfect in the sense of \cite[lemma~3.13]{RW}.
\end{lem}
\begin{proof}
The twisted coinvariant quotient $\widetilde{\Pi}_{(S,\psi)}$ is finite-dimensional, see Roberts and Schmidt \cite[thm.~6.3.2 and lemma~7.1.1]{Roberts-Schmidt_Bessel}.
The coinvariant quotient $\widetilde{\Pi}_S$ factorizes over the Siegel-Jacquet module and is thus finite-dimensional by the theorem of Waldspurger and Tunnell.
Hence $\widetilde{\Pi}$ has finite length by the theory of Gel'fand and Kazhdan \cite[5.12.d--f]{Bernstein-Zelevinsky76}.
Dani\c{s}man has shown that $\widetilde{\Pi}$ is perfect \cite[prop.~4.7]{Danisman}.
\end{proof}

\subsection{Double coset decompositions}

\begin{lem}\label{lem:double_cosets}
There are disjoint double coset decompositions $G=QH=PH\sqcup P\mathbf{s}_2H$.
\end{lem}
\begin{proof}
$H$ acts transitively on $V\setminus \{0\}$, so for every $g\in G$ there is $h\in H$ with $h^{-1}e_1=g^{-1}e_1$, hence $hg^{-1}$ is an element of $Q$.
This shows $G=QH$ and it remains to find representatives of $P\backslash G/H$.
Let $p_{34}:V\to L_0^*$ be the projection with respect to the fixed polarization. Fix $g\in G$ and write $L=g^{-1}L_0$. The dimension of $p_{34} L$ is zero, one or two.

If $\dim p_{34}L=0$, then $L=L_0$ and thus $g\in P\subseteq PH$ by definition of $P$.
If $\dim p_{34}L=1$, then $L$ is generated by $v_1\in L_0+\lambda_1e_1^\ast+\lambda_2e_2^\ast$ and $v_2=\lambda_2e_1-\lambda_1e_2$ for some $(\lambda_1,\lambda_2)\in k^2\setminus \{0\}$.
Since $\widetilde{T}$ acts transitively on $L_0^\ast\setminus \{0\}$, there is $a\in K^\times$ with $\widetilde{t}_av_1\in L_0+e_2^\ast$ and $\widetilde{t}_av_2=e_1$.
This implies $\widetilde{n}\widetilde{t}_av_1=e_2^\ast$ and $\widetilde{n}\widetilde{t}_av_2=e_1$ for suitable $\widetilde{n}\in\widetilde{N}$,
so $\mathbf{s}_2\widetilde{n}\widetilde{t}_a g^{-1}\in P$ and thus $g\in P\mathbf{s}_2 H$.
Finally, if $\dim p_{34}L=2$, then $L$ is generated by $v_1=v_{11}e_1+v_{12}e_2+e_1^\ast$ and $v_2=v_{21}e_1+v_{22}e_2+e_2^\ast$ for certain $v_{ij}\in k$ with $v_{21}=v_{12}$, because $\left<v_1,v_2\right>=0$.
Applying $\widetilde{n}=\left(\begin{smallmatrix}E_2&-A\\0&E_2\end{smallmatrix}\right)\in \widetilde{N}$ with $A=\left(\begin{smallmatrix} v_{22}\alpha & v_{21}\\v_{12} & v_{22}\end{smallmatrix}\right)$
yields $\widetilde{n} v_1=\lambda e_1 + e_1^\ast$ and $\widetilde{n}v_2=e_2^\ast$ with $\lambda=v_{11}-v_{22}\alpha$.
If $\lambda=0$, then $\widetilde{n} v_1$ and $\widetilde{n} v_2$ generate $L_0^\ast$, hence $J\widetilde{n}L= L_0$ and this means $g\in PJ\widetilde{N}\subseteq PH$.
If $\lambda\neq0$, apply $h=\left(\begin{smallmatrix}E_2&0\\-B&E_2\end{smallmatrix}\right)\in H$
with $B=\diag(\lambda^{-1},\lambda^{-1}\alpha)$
to obtain $h\widetilde{n}v_1=\lambda e_1$ and
$h\widetilde{n}v_2=e_2^\ast$.
This implies $\mathbf{s}_2h\widetilde{n}L\subseteq L_0$, hence $g\in P\mathbf{s}_2 H$.

The decomposition is disjoint, because $\dim p_{34}(phL_0)$ is either zero or two for $p\in P$, $h\in H$,
so $\mathbf{s}_2\notin PH$.
\end{proof}

\begin{lem}\label{lem:double_coset_intersections}
$Q\cap H=TZ_G\widetilde{N}$ and $P\cap H=T\widetilde{T}\widetilde{N}=B_H$ and 
\begin{equation*}
 P\cap \mathbf{s}_2H\mathbf{s}_2^{-1} = \left\{\begin{pmatrix}a&b & 0 & 0\\ c&d&0&0\\0&0&d&-c\\0&0&-b&a \end{pmatrix}\in G\mid ad-bc\in k^\times\right\} \cong \Gl(2,k)\ .
\end{equation*}
\end{lem}
\begin{proof}
This is straightforward.
\end{proof}

\section{$H$-functionals give rise to exceptional poles}

Fix the standard maximal compact subgroup $K_H=H\cap \Gl(2,\mathcal{O}_K)$ of $H$.

\begin{lem}\label{lem:KH_sph_in_coinvariants}
Let $\pi=\ind_{B_H}^H(\chi\boxtimes1)$ be the induced representation for an unramified character $\chi$ of $T$ and the trivial character of $\widetilde{T}\widetilde{N}$, then the subspace of $K_H$-invariant vectors in $\pi$ has non-trivial image under the projection $\pi\to \widetilde{\pi}_{(T,\rho_i)}$ for the $T$-characters $\rho_1 = \chi$ and $\rho_2 = \nu^2\chi^{-1}$.
\end{lem}
\begin{proof}
Fix an unramified character $\mu$ of $K^\times$ that coincides with $\nu^{-1}\chi$ on $k^\times\cong T$.
It is easy to see that $\pi$ is isomorphic to the restriction to $H$ of the induced representation $\mu\times\mu^{-1}$ of $\Gl(2,K)$.
We identify $\mu\times\mu^{-1}$ with its Kirillov model $\mathcal{K}\subseteq C_b^\infty(K^\times)$ in the complex vector space of smooth complex-valued functions on $K^\times$ with bounded support.
The subspace $\mathcal{K}^{\Gl(2,\mathcal{O}_K)}$ of $\Gl(2,\mathcal{O}_K)$-invariant vectors is one-dimensional and generated by the spherical Whittaker function $g_0\in \mathcal{K}$,
\begin{equation*}
 g_0(a)= \begin{cases}
               \nu_K^{1/2}\mu(a)-\nu_K^{1/2}\mu^{-1}(a) & \vert a\vert_K\leq1 \ ,\\
               0                                        & \vert a\vert_K>1    \ ,
              \end{cases}
\end{equation*}
where $\nu_K=\nu\circ N_{K/k}$ is the valuation on $K$, compare Bump \cite[\S4.(6.30)]{Bump}.
The unnormalized Jacquet quotient $\mathcal{K}_{\widetilde{N}}$ is given by the quotient of $\mathcal{K}$ by the subspace $C_c^\infty(K^\times)$ of functions with compact support, compare \cite[lemma~3.3]{RW}.
The action of $\widetilde{T}$ on $\mathcal{K}_{\widetilde{N}}$ is trivial because the center acts trivially and $\mu$ is unramified.
Therefore the image of $g_0$ in the $T$-module $\mathcal{K}_{\widetilde{T}\widetilde{N}}$ is determined by its asymptotic behaviour for $a\to 0$.
For small $a$ however, $g_0(a)$ is a non-trivial linear combination of the characters $\rho'_1=\nu^{1/2}_K\mu$ and $\rho'_2=\nu^{1/2}_K\mu^{-1}$ of $T'=\{\diag(a,1)\mid a\in K^\times\}$, so the image of $g_0$ in $(\mathcal{K}_{\widetilde{T}\widetilde{N}})_{(T',\rho'_i)}$ is non-zero for $i=1,2$.
The restriction of $\rho'_i$ to $k^\times$ is $\rho_i$.
This shows that the composition of the inclusions and projections
 $$\mathcal{K}^{\Gl(2,\mathcal{O}_K)}\hookrightarrow\pi^{K_H} \hookrightarrow \pi \to  \widetilde{\pi}_{(T,\rho_i)}\to (\mathcal{K}_{\widetilde{T}\widetilde{N}})_{(T',\rho_i')}$$
is non-zero for $i=1,2$.
\end{proof}

\begin{prop}\label{prop:K_H_spherical_Pi}
Fix an infinite-dimensional irreducible admissible representation $\Pi\in\CCC_G$ and let $\widetilde{\Pi}=\beta_\Lambda(\Pi)$ for $\Lambda=1$.
If there is a non-trivial $H$-equivariant functional $\Pi\to \C$,
then the natural morphism from $\Pi^{K_H}$ to $\widetilde\Pi_{(T,\rho)}$ is non-zero for the $T$-characters $\rho=1$ and $\rho=\nu^2$.
\end{prop}
\begin{proof}
Since $\Pi$ admits a non-trivial $H$-functional, the center $Z_G$ acts trivially. Consider $\Pi$ as a smooth representation of the semisimple group $H/Z_G$ in the natural way.
$\widetilde{\Pi}$ is perfect by lemma~\ref{Bessel_module_perfect}, so $\dim\widetilde{\Pi}_{(T,\chi)}=1$ for every smooth character $\chi$ of $k^{\times}$.
Recall that $\pi=\ind_{B_H}^H(\delta_{B_H})$ is the unique non-split extension $0\to\St_H\to\pi\to 1_H\to 0$
of the trivial representation $1_H$ of $H$ by the special representation $\St_H$ of $H$.
By dual Frobenius reciprocity \cite[I.(35)]{Cartier}
\begin{gather*}
 \dim \Hom_H(\Pi,\pi)
 =\dim \Hom_T(\widetilde{\Pi},\delta_{B_H})=1\ ,
\end{gather*} so there is a non-zero $H$-equivariant morphism $\varphi:\Pi\to \pi$.
We claim that $\varphi$ is surjective.
Indeed, assume it is not.
Then the image of $\varphi$ is the unique $H$-submodule $\St_H$ of $\pi$.
Since $\Pi$ admits a non-trivial $H$-functional, there is an exact sequence in $\CCC_H$
\begin{equation*}
0\to U\to \Pi\to 1_H\oplus \St_H\to 0
\end{equation*}
with $U=\ker(\Pi\to1_H\oplus \St_H)$.
The exact Bessel functor $\beta_\Lambda$ yields an exact sequence in $\CCC_T$
\begin{equation*}
 0\to \widetilde{U}\to \widetilde\Pi\to 1\oplus \nu^2 \to 0\ .
\end{equation*}
For every smooth character $\chi$ of $k^\times$, the left-derived functor of $(T,\chi)$-coinvariants is the functor of $(T,\chi)$-invariants, compare \cite[lemma~A.2]{RW}.
Since $\widetilde{\Pi}$ is perfect, the associated long exact sequence
\begin{equation*}
\xymatrix{
\cdots \ar[r] & \widetilde{\Pi}^{(T,\chi)} \ar[r] & (1\oplus \nu^2)^{(T,\chi)} \ar[r] & \widetilde{U}_{(T,\chi)} \ar[r] & \widetilde{\Pi}_{(T,\chi)} \ar[r] & (1\oplus\nu^2)_{(T,\chi)} \ar[r] & 0\ .
}\end{equation*}
shows
$$\dim_\C \widetilde{U}_{(T,\chi)}=1\ .$$
By the same argument as before, there is a non-zero morphism of $H$-modules $$\varphi_2:U\to \pi\ .$$
\begin{enumerate}
\item If $\varphi_2$ is not surjective, its image is the unique submodule $\St_H$ of $\pi$.
Then $\Pi'=\Pi/\ker\varphi_2$ is an extension in $\CCC_{H/Z_G}$
\begin{equation*}
\xymatrix{
0\ar[r] & \St_H\ar[r] & \Pi'\ar[r] & 1_H \oplus \St_H \ar[r] & 0\ .
}
\end{equation*}
Since $\dim\Ext_{\CCC_{H/Z_G}}(1_H\oplus \St_H,\St_H)=1$ by \cite[thm.~1]{Orlik_Ext}, compare \cite[prop.~9]{Schneider-Stuhler}, this extension is either split or isomorphic to $\St_H\oplus\ind_{B_H}^H(\delta_{B_H})$.
In both cases the exact Bessel functor
yields a surjection $\widetilde{\Pi}\twoheadrightarrow\widetilde{\Pi'}\cong\nu^2\oplus1\oplus\nu^2$ of $T$-modules.
This is a contradiction to $\dim\widetilde{\Pi}_{(T,\nu^2)}=1$.
\item If $\varphi_2$ is surjective, let $\Pi''=\Pi/ \varphi_2^{-1}(\St_H)$ be the quotient by the preimage of the submodule $\St_H$ in $\pi$.
It is an extension in $\CCC_{H/Z_G}$
\begin{equation*}
\xymatrix{ 0\ar[r]& 1_H\ar[r]& \Pi''\ar[r]& 1_H\oplus \St_H\ar[r]& 0 \ .}
\end{equation*}
Since $\dim\Ext_{\CCC_{H/Z_G}}(1_H\oplus \St_H,1_H)=1$ by \cite[thm.~1]{Orlik_Ext}, compare \cite[prop.~8]{Schneider-Stuhler},
this extension is either split or isomorphic to
$1_H\oplus \ind_{B_H}^H(1_{B_H})\ .$
The exact Bessel functor  yields a surjection $\widetilde{\Pi}\to \widetilde{\Pi''}\cong 1\oplus1\oplus\nu^2$ of $T$-modules,
which contradicts $\dim \widetilde{\Pi}_{(T,1)}=1$.
\end{enumerate}
This shows our claim that $\varphi$ is surjective.
Exactness of the Bessel functor and the functor of $K_H$-invariants and right-exactness of the $(T,\rho)$-coinvariant functor yield a commutative diagram with surjective vertical arrows
\begin{equation*}
 \xymatrix{
\Pi^{K_H}\ar@{^{(}->}[r]\ar@{->>}[d] & \Pi \ar[r]\ar@{->>}[d]_\varphi & \widetilde{\Pi}_{(T,\rho)}\ar@{->>}[d] \\
\pi^{K_H}\ar@{^{(}->}[r]\ar@/_1pc/[rr]_{\neq0}                & \pi \ar[r]     & \widetilde\pi_{(T,\rho)} \ .
 }
\end{equation*}
The composition of the lower horizontal arrows is non-zero by lemma~\ref{lem:KH_sph_in_coinvariants}.
Both objects on the right hand side are one-dimensional, so the right vertical arrow is actually an isomorphism. 
Hence the composition of the upper horizontal arrows is non-zero.
\end{proof}

Piatetski-Shapiro \cite[thm.\,4.5]{PS-L-Factor_GSp4} has shown that if $L_{\ex}^{\PS}(s,\Pi,\Lambda,\mu)$ admits a pole at $s_0\in\C$, then there is a non-trivial $H$-equivariant functional $\mu\otimes\Pi\to(\nu^{-s_0-1/2}\circ\det)$. We show that in the case of anisotropic Bessel models the converse holds as well.

\begin{thm}\label{thm:exceptional_poles}
Fix an infinite-dimensional irreducible admissible representation $\Pi$ of $G$ and a smooth character $\mu$ of $k^\times$.
If there is a non-trivial $H$-equivariant functional $\mu\otimes\Pi\to(\rho\circ\det)$ for an unramified character $\rho$ of $k^\times$, then $\Pi$ admits an anisotropic Bessel model $(\Lambda,\psi)$ with $\Lambda=(\mu^{-1}\rho)\circ N_{K/k}$.
The exceptional $L$-factor $L^{\PS}_{\ex}(s,\Pi,\Lambda,\mu)$ is either
$L(s,\rho\nu^{1/2})$ or $L(s,\rho\nu^{1/2})L(s,\chi_{K/k}\rho\nu^{1/2})$
with the quadratic character $\chi_{K/k}$ attached to $K/k$.
\end{thm}
\begin{proof}
By a twist we can assume $\mu=1$.
By proposition~\ref{prop:nec_cond}, $\Lambda$ yields an anisotropic Bessel model.
Recall that $\Phi_0\in C_c^\infty(K^2)$ is the characteristic function of $\mathcal{O}_K\times\mathcal{O}_K$.

Iwasawa decomposition of $H\ni h=\widetilde{n}x_\lambda \widetilde{t}_ak\in \widetilde{N}T\widetilde{T}K_H$ yields $W_v(h)=\Lambda(a)W_{v}(x_\lambda k)$ and $\Phi_0((0,1)h)=\Phi_0((0,1) \widetilde{t}_a k)=\Phi_0((0,\overline{a}))$.
Up to a constant volume factor the zeta function $Z^{\PS}(s,v,\Lambda,\Phi_0,1)$ equals
\begin{gather*}
\int_T\int_{K_H} W_{v}(x_\lambda k)|\lambda|^{s-3/2}\Diff k\Diff^\times\! x_\lambda 
 \int_{K^\times}\Phi_0((0,\overline{a})) \Lambda(a)|a\overline{a}|^{s+1/2}  \Diff^\times\! a\\
 = \mathrm{const.}\ Z^{\PS}_{\reg}(s,W_v^{\mathrm{av}},1) L(s+1/2,\Lambda)\ ,
\end{gather*}
where $W_v^{\mathrm{av}}(x_\lambda)=\int_{K_H}W_{v}(x_\lambda k) \Diff k/\mathrm{vol}(K_H)$ is the $K_H$-average of the Bessel function $W_v$, compare \cite[prop.~2.5]{Danisman}.
Since the exceptional factor $L_{\ex}^{\PS}(s,\Pi,\Lambda,1)$ is the regularization $L$-factor of the family
\begin{equation*}
 Z^{\PS}(s,v,\Phi_0,\Lambda,1)/L^{\PS}_{\reg}(s,\Pi,\Lambda,1)\quad,\qquad v\in \Pi\ ,
\end{equation*}
it divides $L(s+1/2,\Lambda) = L(s,\rho\nu^{1/2})L(s,\chi_{K/k}\rho\nu^{1/2})$.
It remains to be shown that $L(s,\rho\nu^{1/2})$ always divides the exceptional $L$-factor.
We can assume $\rho=1$ by an unramified twist.
It is sufficient to show that the $T$-equivariant functional
\begin{equation*}
l:\Pi\to\nu^{2}, \qquad v\mapsto \lim_{s\to -1/2} Z^{\PS}_{\reg}(s,W_v,1)/L^{\PS}_{\reg}(s,\Pi,\Lambda,1)
\end{equation*}
is non-zero on some $v\in \Pi^{K_H}$, since then $W_v=W_v^{\mathrm{av}}$ and we obtain a pole of $L_{\ex}^{\PS}(s,\Pi,\Lambda,1)$ at $s=-1/2$.
Indeed, $l$ is well-defined on $\Pi$ by construction of the regular $L$-factor.
It is easy to see that $v\mapsto W_v(x_\lambda)$ defines a non-trivial $TS$-equivariant morphism $\Pi\to C_b^\infty(k^\times)$ to the $TS$-module $C_b^\infty(k^\times)$ defined in \cite[example~3.1]{RW}. This morphism factorizes over the perfect Bessel module $\widetilde{\Pi}$ and yields an embedding $\widetilde{\Pi}\hookrightarrow C_b^\infty(k^\times)$ \cite[lemma~3.24]{RW}.
Hence $l$ spans the one-dimensional space $\Hom_{T}(\widetilde{\Pi},\nu^{2})$ \cite[lemma~3.31]{RW}.
Proposition~\ref{prop:K_H_spherical_Pi} provides the required $v\in\Pi^{K_H}$.
\end{proof}

\begin{cor}\label{cor:final}
Fix a non-cuspidal infinite-dimensional irreducible admissible representation $\Pi$ of $G$ that admits an anisotropic Bessel model $(\Lambda,\psi)$ attached to $K/k$.
Then the character $\Lambda$ and the exceptional $L$-factor $L^{\PS}_{\ex}(s,\Pi,\Lambda,\mu)$ are given by table~\ref{tab:exceptional_poles2}.
\end{cor}
\begin{proof}
If the exceptional factor contains a Tate-factor $L(s,\nu^{1/2}\rho)$ with an unramified character $\rho$ of $k^\times$,
then $\Pi$ admits a non-zero $(H,\rho\circ\lambda_G)$-equivariant functional, see \cite[thm.~4.2]{PS-L-Factor_GSp4}.
These functionals are classified in theorem~\ref{thm:classification_H-functionals}.
If such a functional exists, the character $\Lambda$ is uniquely determined
\cite[thm.~6.2.2]{Roberts-Schmidt_Bessel}, so $\Lambda=\rho\circ N_{K/k}$.
Further, the exceptional factor contains the Tate factor $L(s,\nu^{1/2}\rho)$ exactly once by theorem~\ref{thm:exceptional_poles}.
\end{proof}

\section{Classification of $H$-functionals}
We classify the non-cuspidal irreducible admissible representations $\Pi$ of $G$ that admit a non-trivial $H$-equivariant functional $\Pi\to (\rho\circ\lambda_G)$ for a smooth character $\rho$ of $k^\times$. Such $\Pi$ are all non-generic by the following result of Piatetskii-Shapiro.

\begin{lem}\label{lem:generic_no_Hfctls}
Generic irreducible admissible representations $\Pi$ of $G$ do not admit non-zero $(H,\rho\circ\det)$-equivariant functionals for any smooth character $\rho$ of $k^\times$.
\end{lem}
\begin{proof}
This has been shown by Piatetski-Shapiro \cite[thm.\,4.3]{PS-L-Factor_GSp4}, compare \cite[\S5]{Schmidt_Tran}.
\end{proof}

\begin{lem}\label{lem:nec_conds}
For an infinite-dimensional non-generic irreducible admissible $\Pi\in\CCC_G$ and a smooth character $\rho$ of $k^\times$,
let $\widetilde{\Pi}=\beta_\Lambda(\Pi)$ for $\Lambda=\rho\circ N_{K/k}$.
If $\Hom_T(\widetilde{\Pi},\rho)$ is non-zero, %
then $\Pi$ admits a Bessel model $(\Lambda,\psi)$ for $\Lambda=\rho\circ N_{K/k}$.
\end{lem}

\begin{proof}
Assume that $\Lambda$ does not provide a Bessel model for $\Pi$, i.e.\ $\widetilde{\Pi}_{(S,\psi)}=0$.
Then $\widetilde\Pi$ is a trivial $S$-module by the functorial exact sequence of \cite[5.12.d]{Bernstein-Zelevinsky76}.
For the unnormalized Siegel-Jacquet module $J_P(\Pi)$, transitivity
of the coinvariant functors yields an isomorphism of non-zero vector spaces
\begin{equation*}
\widetilde{\Pi}_{(T,\rho)} \quad \cong \quad J_P(\Pi)_{(T\widetilde{T},\rho\boxtimes\Lambda)}\ .
\end{equation*}
Especially, $J_P(\Pi)$ is non-zero. 
Without loss of generality we can now assume that $\Pi$ is normalized as in \cite[table 1]{RW} by a suitable twist.
The center $Z_G$ is contained in $\widetilde{T}$, so $\widetilde{\Pi}\neq0$ implies that the central character is $\omega=\rho^2$.
This condition and the list of constituents in $J_P(\Pi)$ given by \cite[table A.3]{Roberts-Schmidt} imply
$J_P(\Pi)_{(T,\rho)}=0$ in every case except for
type \nosf{IVc} with $\rho=1$ and type \nosf{IIb} with $\chi_1^{\pm1}=\rho\nu^{-3/2}$ and $\rho\neq 1=\rho^2$.
In case \nosf{IVc} the $T$-coinvariant space of the Siegel-Jacquet module is the special representation
$J_P(\Pi)_{(T,1)}\cong \St\boxtimes 1$ of $M\cong \Gl(2,k)\times \Gl(1,k)$,
so the $\widetilde{T}$-coinvariant quotient is zero by the theorem of Waldspurger and Tunnell.
For the remaining case of type \nosf{IIb}, the $\widetilde{T}$-module $J_P(\Pi)_{(T,\rho)}$ has only trivial constituents, so its $(\widetilde{T},\Lambda)$-coinvariant space is zero for $\Lambda\neq1$.
This shows that $\Lambda$ must provide a Bessel model.
\end{proof}

\begin{lem}\label{IVb_noHfctl}
For $\Pi$ of type \nosf{IVb, VIIIb, IXb}, there is no non-zero $(H,\rho\circ\lambda)$-equivariant functional.
In other words, $\Hom_H(\Pi,\rho\circ\lambda_G)=0$ for every smooth character $\rho$ of $k^\times$.
\end{lem}
\begin{proof}
We first discuss the case of type~\nosf{IVb}.
By a possible twist we can assume that $\Pi$ is normalized as in \cite[table~1]{RW}.
By lemma~\ref{lem:nec_conds} we only have to consider the case where $\Lambda=\rho\circ N_{K/k}$ provides a Bessel model.
Recall that $\Pi$ is a quotient of the unnormalized Klingen induced representation
$$I=\ind_Q^G(\nu^4\boxtimes \St(\nu^{-2}))= \nu^2\rtimes \St(\nu^{-1})\ $$
with respect to the usual factorization of the Levi subgroup as in \cite{RW} or \cite{Roberts-Schmidt}.
A hypothetical $H$-functional of $\Pi$ would give rise to one of $I$. The double coset decomposition of lemma~\ref{lem:double_cosets} yields an isomorphism of $H$-modules $$I|_H\cong \ind_{H\cap Q}^H(\nu^4\boxtimes \St(\nu^{-2}))\ .$$
Frobenius reciprocity in the sense of lemma~\ref{lem:Frobenius} implies
\begin{gather*}
 \Hom_H(I|_H,\rho\circ\lambda_G)
 \cong \Hom_{H\cap Q}(\nu^4\boxtimes \St(\nu^{-2}),(\rho\circ\lambda_G)\otimes\delta_{H\cap Q})\ .
\end{gather*} %
It remains to be shown that this space is zero.
$H\cap Q$ contains the affine linear group
\begin{equation*}\left\{\begin{pmatrix}
a&0&b\alpha & 0\\
0&a&0&b\\
0&0&1&0\\
0&0&0&1
\end{pmatrix}\mid a\in k^\times, b\in k\right\}
\cong \Gl_a(1,k)\ .
\end{equation*}
The modulus character of $H\cap Q$ reduced to $\Gl_a(1,k)$ is $\delta_{H\cap Q}|_{\Gl_a(1)}=i_*(\nu^2)$ with the functor $i_*$ defined in \cite{RW}.
We obtain an embedding
\begin{gather*}
 \Hom_{H\cap Q}(\nu^4\boxtimes \St(\nu^{-2}), (\rho\circ\lambda_G)\otimes\delta_{H\cap Q})
 \hookrightarrow
 \Hom_{\Gl_a(1)}(\nu^4\boxtimes \St(\nu^{-2})|_{\Gl_a(1,k)}, i_*(\nu^2\rho)) \ .
\end{gather*}
With the functor $i^*$ left-adjoint to $i_*$ one shows that the right hand side is isomorphic to $\Hom_{\Gl(1)}(\nu^3,\nu^2\rho)$.
By assumption $\rho\circ N_{K/k}$ provides a Bessel model,
so $\nu^3\neq\nu^2\rho$ and the statement follows.

Every $\Pi$ of type \nosf{VIIIb} or \nosf{IXb} is a quotient of a Klingen induced representation $I=\chi\rtimes\pi$ with a smooth quadratic character $\chi$ of $k^\times$ and a cuspidal irreducible $\pi\in\CCC_{\Gl(2,k)}$.
The rest of the argument is analogous and uses that the Jacquet module $i^*(\nu\chi\otimes\pi|_{\Gl_a(1)})=0$ vanishes.
\end{proof}

\begin{prop}\label{prop:nec_cond}
Fix an infinite-dimensional non-cuspidal irreducible admissible $\Pi\in\CCC_G(\omega)$ that
admits a non-trivial $H$-equivariant functional $\Pi\to(\rho\circ\lambda_G)$ for a smooth character $\rho$ of $k^\times$.
Then $\Pi$ is non-generic and admits an anisotropic Bessel model $(\Lambda,\psi)$ with $\Lambda=\rho\circ N_{K/k}$. Further, $\Pi$ belongs to type \nosf{IIb, Vb, Vc, Vd, VIb} or \nosf{XIb}.
\end{prop}

\begin{proof}
$\Pi$ is non-generic by lemma~\ref{lem:generic_no_Hfctls}.
The inclusion $$\Hom_H(\Pi,\rho\circ\lambda_G)\hookrightarrow \Hom_{B_H}(\Pi,\rho\circ\lambda_G)\cong \Hom_{T}(\widetilde{\Pi},\rho)$$ and lemma~\ref{lem:nec_conds}
show that the Bessel model exists.
For non-cuspidal $\Pi$ the classification of anisotropic Bessel models \cite[theorem~6.2.2]{Roberts-Schmidt_Bessel} implies that $\Pi$ belongs to the indicated type or to the types already excluded in lemma~\ref{IVb_noHfctl}.
\end{proof}

\begin{lem}\label{lem:Mackey}
For irreducible admissible representations $\pi$ of $\Gl(2,k)$ and $\chi$ of $k^\times$, the Siegel induced representation $I=\pi\rtimes\chi=\ind_P^G((\pi\boxtimes\chi)\otimes\delta_P^{1/2})$ admits
an exact sequence of smooth $H$-modules, functorial in $\pi\boxtimes\chi$,
\begin{equation*}
  0\longrightarrow I_1   \longrightarrow I \longrightarrow I_0 \longrightarrow 0
\end{equation*}
with $I_1\cong \ind_{\mathbf{s}_2^{-1}P\mathbf{s}_2\cap H}^H((\pi\boxtimes\chi)\otimes\delta_P^{1/2})^{\mathbf{s}_2}$ and $I_0\cong \ind_{P\cap H}^H((\pi\boxtimes\chi)\otimes\delta_P^{1/2})\ $.
\end{lem}
\begin{proof} This is theorem~5.2 of Bernstein and Zelevinski \cite{Bernstein-Zelevinsky77} applied to the double coset decomposition of lemma~\ref{lem:double_coset_intersections}.
Dimension estimates show that $P\mathbf{s}_2 H$ is open in $G$.
\end{proof}

\begin{lem}\label{lem:H-fctls_I_?}
The $(H,\rho)$-coinvariant quotient of $I_0$ has dimension 
 \begin{equation*}
\dim(I_0)_{(H,\rho)} = \dim \Hom_{\widetilde{T}}(\pi\boxtimes\chi, \rho\circ\lambda_G) \cdot \dim \Hom_{T}(\omega_\pi\chi\nu^{3/2},\rho\nu^2)\leq 1\ ,
\end{equation*} where $\omega_\pi$ is the central character of $\pi$.
The $(H,\rho)$-coinvariant quotient of $I_1$ has dimension 
 \begin{equation*}
\dim(I_1)_{(H,\rho)} = \dim \Hom_{\Gl(2,k)}(\chi\otimes\pi, \rho\circ\det) = 
\begin{cases}1 & \pi\cong(\chi^{-1}\rho)\circ\det\ ,\\ 0 & \text{otherwise}\ .\end{cases}
\end{equation*}
\end{lem}
\begin{proof}
Lemma~\ref{lem:Frobenius} implies
\begin{gather*}
\dim(I_0)_{(H,\rho)}=\dim \Hom_H(\ind_{H\cap P}^H((\pi\boxtimes\chi)\otimes\delta_P^{1/2}),\rho\circ\lambda_G)\\
=\dim \Hom_{H\cap P}(((\pi\boxtimes\chi)\otimes\delta_P^{1/2}),\delta_{P\cap H}\otimes\rho\circ\lambda_G)\ .
\end{gather*}
Note that $H\cap P=T\widetilde{T}\widetilde{N}$ by lemma~\ref{lem:double_coset_intersections} and that the action of $\widetilde{N}$ is trivial on both sides.
The modulus factors are $\delta_{P}(x_\lambda \widetilde{t}\widetilde{n})=\vert\lambda\vert^3$ and $\delta_{P\cap H}(x_\lambda\widetilde{t}\widetilde{n})=\vert\lambda\vert^2$ for $x_\lambda\in T$, $\widetilde{t}\in \widetilde{T}$, $\widetilde{n}\in\widetilde{N}$.

For the second assertion, note that $\mathbf{s}_2H\mathbf{s}_2^{-1}\cap P\cong \Gl(2,k)$ by lemma~\ref{lem:double_coset_intersections},
so the modulus character $\delta_{\mathbf{s}_2H\mathbf{s}_2^{-1}\cap P}$ is trivial.
It is easy to see that $\delta_P$ restricted to $\mathbf{s}_2H\mathbf{s}_2^{-1}\cap P$ is also trivial.
The rest of the argument is analogous to the first assertion.
\end{proof}
\begin{table}\caption{Non-zero $(H,\rho\circ\lambda_G)$-functionals for non-cuspidal irreducible $\Pi$.\label{tab:sigma}}
\begin{center}
\begin{small}
\begin{tabular}{lllll}
\toprule
Type       & $\Pi\in\CCC_G$                      & Condition                  & $\rho$ \\
\midrule
\nosf{IIb} & $(\chi_1\circ\det)\rtimes\sigma$    &                            & $\chi_1\sigma$  \\
\nosf{Vb}  & $L(\nu^{1/2}\xi \St,\nu^{-1/2}\sigma)$& $\xi \neq \chi_{K/k}$    & $\sigma$  \\
\nosf{Vc}  & $L(\nu^{1/2}\xi \St,\nu^{-1/2}\xi\sigma)$& $\xi \neq \chi_{K/k}$ & $\xi\sigma$  \\
\nosf{Vd}  & $L(\nu\xi,\xi\rtimes\nu^{-1/2}\sigma)$& $\xi = \chi_{K/k}$       & $\sigma$, $\xi\sigma$   \\
\nosf{VIb} & $\tau(T,\nu^{-1/2}\sigma)$          &                            & $\sigma$ \\
\nosf{XIb} & $L(\nu^{1/2}\pi,\nu^{-1/2}\sigma)$  & $\pi_{\widetilde{T}}\neq0$ & $\sigma$ \\
\bottomrule
\end{tabular}
\end{small}
\end{center}
\end{table}

\begin{thm}\label{thm:classification_H-functionals}
For an infinite-dimensional non-cuspidal irreducible representation $\Pi\in\CCC_G$ and a smooth character $\rho$ of $k^\times$, the dimension of $(H,\rho)$-equivariant functionals is
\begin{equation*}
 \dim\Hom_{H}(\Pi,\rho\circ \lambda_G)=
 \begin{cases}
  1 & \text{if $\Pi$ and $\rho$ occur in table~\ref{tab:sigma} and the condition holds} ,\\
  0 & \text{otherwise}\ .
 \end{cases}
\end{equation*}
\end{thm}

\begin{proof}
Lemma~\ref{lem:nec_conds} yields
$\dim\Hom_H(\Pi,\rho\circ\lambda_G)\leq\dim\Hom_{B_H}(\Pi,\rho\circ\lambda_G)=\dim\widetilde{\Pi}_{(T,\rho)}$.
By lemma~\ref{Bessel_module_perfect} and the uniqueness of Bessel models \cite[thm.~6.3.2]{Roberts-Schmidt_Bessel}, the dimension is bounded by $\dim\widetilde{\Pi}_{(T,\rho)}=\dim\widetilde{\Pi}_{S,\psi}\leq1$.

By proposition~\ref{prop:nec_cond}, an $H$-functional can only exist if $\Pi$ is of the indicated type and $\Lambda=\rho\circ N_{K/k}$ yields a Bessel model for $\Pi$.
For these types it remains to be shown that $H$-functionals exist for every $\rho$ in the table and no other $\rho$.
By a twist we can assume that $\Pi$ is normalized as in \cite[table~1]{RW}, i.e.\ we assume $\sigma=1$ except for type \nosf{IIb} where we set $\sigma\chi_1=1$.
For each case we consider a suitable Siegel induced representation $I$.
Right-exactness of the coinvariant functor then gives an exact sequence with first and third term given by lemma~\ref{lem:H-fctls_I_?}:
\begin{equation*}
 \cdots\longrightarrow (I_1)_{(H,\rho)}\longrightarrow I_{(H,\rho)} \longrightarrow (I_0)_{(H,\rho)}\longrightarrow 0\ .
\end{equation*} 
For type \nosf{IIb}, $\Pi$ is a quotient of $I=\nu^{1/2}(\chi_1\times\chi_1^{-1})\rtimes\nu^{-1/2}$ \cite[table~2]{RW}.
Right exactness and lemma~\ref{lem:H-fctls_I_?} imply $\dim I_{(H,\rho)}=\dim (I_0)_{(H,\rho)}=\dim \Hom_{T}(\nu^2,\rho\nu^2)$.
The kernel of the projection $I\to\Pi$ is generic of type \nosf{IIa} and does not admit any $H$-functional, so
$\dim I_{(H,\rho)} = \dim\Pi_{(H,\rho)}$. Hence $\Pi$ admits an $(H,\rho\circ\lambda_G)$-functional exactly for $\rho=1$.

If $\Pi$ is of type \nosf{Vb}, it is a quotient of $I=\St(\nu^{1/2}\xi)\rtimes\nu^{-1/2}$ and the argument is analogous.
The condition $\xi\neq\chi_{K/k}$ implies $\Hom_{\widetilde{T}}(\St(\xi),1)\neq0$ by the theory of Waldspurger and Tunnell.
Type \nosf{Vc} is a twist of \nosf{Vb}.
For type \nosf{XIb} choose $I=\nu^{1/2}\pi \rtimes\nu^{-1/2}$ and apply the analogous argument.

$\Pi$ of type \nosf{VIb} is a submodule of $I=(\nu^{1/2}\circ\det) \rtimes\nu^{-1/2}$. Right-exactness and lemma~\ref{lem:H-fctls_I_?} show
that $I$ admits an $(H,\rho\circ\lambda_G)$-functional for $\rho=1$.
On the other hand, $\Pi$ is a quotient of $I=(\nu^{-1/2}\circ\det) \rtimes\nu^{1/2}$, so $\dim \Pi_{(H,\rho)} \leq \dim I_{(H,\rho)}=0$ for $\rho\neq1$.

$\Pi$ of type \nosf{Vd} is a quotient of $I=(\nu^{1/2}\xi\circ\det)\rtimes \nu^{-1/2}$ with $\xi=\chi_{K/k}$.
By lemma~\ref{lem:H-fctls_I_?},
$(I_0)_{(H,\rho)}$ is non-zero exactly for $\rho=1$ and $\dim(I_1)_{(H,\rho)}$ is non-zero exactly for $\rho=\xi$. 
The kernel of $I\to\Pi$ is of type \nosf{Vc} and does not admit any $(H,\rho)$-functional for $\xi=\chi_{K/k}$.
Therefore $\dim I_{(H,\rho)}=\dim\Pi_{(H,\rho)}$ is non-zero for $\rho=1$ and is zero for every smooth $\rho\neq1,\xi$.
Since $\Pi\cong\xi\otimes\Pi$, we can apply the analogous argument to $\xi\otimes I$ instead of $I$ 
and this implies the statement.
\end{proof}

\section{Conclusion}\label{s:tables}
We review the results of \cite{Danisman, Danisman_Annals, Danisman2, Danisman3, RW, W_Excep, Subregular}.

\textbf{Table~\ref{tab:exceptional_poles2}}
lists the smooth characters $\Lambda$ of $K^\times$ that yield anisotropic Bessel models \cite{Roberts-Schmidt_Bessel} and the associated regular and exceptional $L$-factors.
For non-cuspidal irreducible $\Pi\in {\cal C}_G$ we use the classification symbols introduced by Sally and Tadi\'{c} \cite{Sally-Tadic} and Roberts and Schmidt \cite{Roberts-Schmidt}.
The term ``all'' means ``all smooth $\Lambda$ that coincide on $k^\times$ with the central character $\omega$ of $\Pi$''.
The regular $L$-factor has been determined by Dani\c{s}man~\cite{Danisman, Danisman2, Danisman3} in the context of local function fields
and analogous arguments yield the corresponding results for local number fields.
For non-cuspidal $\Pi$ the exceptional factor is given by corollary~\ref{cor:final}.
For cuspidal irreducible $\Pi\in\CCC_G$ and odd residue characteristic of $k$, the $L$-factor has been determined by Dani\c{s}man~\cite{Danisman_Annals} in terms of the anisotropic theta lift:
Every non-generic cuspidal $\Pi$ occurs as an anisotropic theta-lift $\theta(\pi_1^{JL},\pi_2^{JL})$
from a pair of finite-dimensional irreducible representations $\pi_1,\pi_2$ of $\CCC_{\Gl(2)}(\omega)$ in the discrete series.
This pair $(\pi_1,\pi_2)$ is uniquely determined by $\Pi$ up to permutation.
For typographical reasons we set $\mu=1$.

\textbf{Table~\ref{tab:spinor_factors}}
lists the full $L$-factors $L^{\PS}(s,\Pi,\Lambda,1)$ attached to arbitrary Bessel models \cite{Roberts-Schmidt_Bessel}.
Split Bessel models are determined by characters $\Lambda=\rho\boxtimes\rho^\divideontimes$ of 
$k^\times\times k^\times$ with $\rho^\divideontimes=\omega/\rho$.
The regular and exceptional $L$-factors for split Bessel models have been determined by the authors \cite{RW, W_Excep, Subregular}.
We thus show that the $L$-factor of Piatetskii-Shapiro is indeed independent of the choice of any Bessel model and coincides with the Galois factor defined by Roberts and Schmidt, see table~A.8 of \cite{Roberts-Schmidt}.
Our results are in accordance with the local Langlands correspondence for $\GSp(4)$ as shown by Gan and Takeda \cite{Gan-Takeda}. Note that they use a different construction of spinor $L$-factors.

\textbf{Local endoscopic $L$-packets}
$\{\Pi_\pm(\pi_1,\pi_2)\}$ in the sense of \cite{W_Endo} are attached to pairs of unitary generic irreducible $\pi_1,\pi_2\in\CCC_{\Gl(2,k)}(\omega)$ with common central character $\omega$, see table~\ref{tab:local_lift}.
The generic constituent $\Pi_+(\pi_1,\pi_2)$ has spinor $L$-factor
\begin{equation*}
 L^\PS(s,\Pi_+(\pi_1,\pi_2),\Lambda,\mu)=L(s,\pi_1\otimes\mu)L(s,\pi_2\otimes\mu)
\end{equation*}
for each Bessel model of $\Pi_+$.
If $\pi_1$ and $\pi_2$ are in the discrete series, there is also a non-generic constituent $\Pi_-(\pi_1,\pi_2)$.
Every Bessel model attached to it is anisotropic. If $k$ has odd residue characteristic,
the spinor $L$-factor is
\begin{equation*}
 L^\PS(s,\Pi_-(\pi_1,\pi_2),\Lambda,\mu)=L(s,\pi_1\otimes\mu)L(s,\pi_2\otimes\mu)\ .
\end{equation*}

\begin{table}
\centering
\begin{small}
 \begin{tabular}{llllll}
\toprule
$\pi_{1}$         & $\pi_{2}$                & $\Pi_+(\pi_1,\pi_2)$                                    &Type & $\Pi_-(\pi_1,\pi_2)$ & Type\\\midrule
 $\mu_1\times\mu_2$ & $\mu_3\times\mu_4$     & $\mu_3\mu_1^{-1}\times\mu_4\mu_1^{-1}\rtimes\mu_1$       &\nosf{I}    & ---                    &\\
 $\mu_1\times\mu_2$ & $\mu\cdot\St $         & $\mu\mu_1^{-1}\cdot \St_{\Gl(2)}\rtimes\mu_1$             &\nosf{IIa}  & ---                    &\\
 $\mu_1\times\mu_2$ & cuspidal                & $\mu_3^{-1}\cdot\pi_2\rtimes\mu_3$                      &\nosf{X}    & ---                    &\\
$\mu\cdot\St$     & $\mu\cdot \St$            & $\tau(S,\nu^{-1/2}\mu)$                                 &\nosf{VIa}  & $\tau(T,\nu^{-1/2}\mu)$                          &\nosf{VIb}\\
$\xi\mu\cdot\St$  & $\mu\cdot \St$            & $\delta(\xi\nu^{1/2}\cdot\St\rtimes\mu\nu^{-1/2})$      &\nosf{Va}   & $\theta_-(\xi\mu \cdot\St,\mu\cdot\St)$ cuspidal &\nosf{Va*}\\
cuspidal          & $\mu\cdot \St$            &$\delta(\mu^{-1}\nu^{1/2}\cdot\Pi_1\rtimes\mu\nu^{-1/2})$&\nosf{XIa}  & $\theta_-(\pi_1,\mu\cdot \St)$          cuspidal &\nosf{XIa*}\\
cuspidal          & cuspidal $\cong\pi_1$     & $\tau(S,\pi_1)$                                         &\nosf{VIIIa}& $\tau(T,\pi_1)$                                  &\nosf{VIIIb}\\
cuspidal          & cuspidal $\not\cong\pi_1$ & cuspidal                                                &            & $\theta_-(\pi_1,\pi_2)$ cuspidal                 &\\
\bottomrule
\end{tabular}
\end{small}
\caption{$L$-packet $\{\Pi_\pm(\pi_1,\pi_2)\}$ attached to endoscopic lifts.\label{tab:local_lift}}
\end{table}

\textbf{Arthur packets for Saito-Kurokawa lifts} $\{\Pi^\mathrm{SK}_{\pm}(\pi)\}$ in the sense of Piatetski-Shapiro \cite{PS-SK}
are attached to unitary generic irreducible $\pi\in\CCC_{\Gl(2)}$ with trivial central character, see table~\ref{tab:SK_local}.
The constituents have been determined explicitly by Schmidt \cite{Schmidt_local_SK}.
For each Bessel model, the non-tempered constituent $\Pi^\mathrm{SK}_{+}(\pi)$ has spinor $L$-factor
\begin{equation*}
 L^\PS(s,\Pi^\mathrm{SK}_{+}(\pi),\Lambda,\mu) = L(s,\pi\otimes\mu)L(s,\nu^{1/2}\mu)L(s,\nu^{-1/2}\mu)\ .
\end{equation*}
The tempered constituent $\Pi^\mathrm{SK}_{-}(\pi)$ only exists for $\pi$ in the discrete series and
is isomorphic to $\Pi_-(\pi,\St)$.
If $k$ has odd residue characteristic, for each Bessel model the spinor $L$-factor is
\begin{equation*}
 L^\PS(s,\Pi^\mathrm{SK}_{-}(\pi),\Lambda,\mu) = L(s,\pi\otimes\mu)L(s,\nu^{1/2}\mu)\ .
\end{equation*}

\begin{table}
\centering
\begin{small}
 \begin{tabular}{llllll}
\toprule
$\pi$   & $\Pi^{\mathrm{SK}}_+(\pi)$                          & Type         & $\Pi^{\mathrm{SK}}_-(\pi)$            & Type  \\\midrule
$\mu\times\mu^{-1}$    & $(\mu\circ\det)\rtimes\mu^{-1}$      & \nosf{IIb}   & ---                                   &       \\
$St$                   & $L(\nu^{1/2}   \cdot \St,\nu^{-1/2})$& \nosf{VIc}   & $\tau(T,\nu^{-1/2})$                  &\nosf{VIb}  \\
$\xi \cdot St$         & $L(\nu^{1/2}\xi\cdot \St,\nu^{-1/2})$& \nosf{Vbc}   & $\theta_-(\xi \cdot\St,\St)$ cuspidal &\nosf{Va*}   \\
cuspidal               & $L(\nu^{1/2}\cdot\pi,\nu^{-1/2})$    & \nosf{XIb}   & $\theta_-(\pi,\St)$          cuspidal &\nosf{XIa*}  \\
\bottomrule
\end{tabular}
\caption{Arthur packet $\{\Pi^{\mathrm{SK}}_\pm(\pi)\}$ attached to Saito-Kurokawa lifts.\label{tab:SK_local}}
\end{small}
\end{table}

\begin{footnotesize}
\bibliographystyle{amsalpha}

\end{footnotesize}
 
\vskip 12 pt
\begin{footnotesize}
\centering{Mirko R\"osner\\ Mathematisches Institut, Universit\"at Heidelberg\\ Im Neuenheimer Feld 205, 69120 Heidelberg\\ email: mroesner@mathi.uni-heidelberg.de}

\vskip 8 pt
\centering{Rainer Weissauer\\ Mathematisches Institut, Universit\"at Heidelberg\\ Im Neuenheimer Feld 205, 69120 Heidelberg\\ email: weissauer@mathi.uni-heidelberg.de}

\end{footnotesize}

\begin{landscape}
\begin{table}\caption{Spinor $L$-factors for anisotropic Bessel models.\label{tab:exceptional_poles2}}
\begin{footnotesize}
\begin{center}
\begin{tabular}{lllll}
\toprule
Type        & $\Pi\in\CCC_{\GSp(4,k)}$  &  $\Lambda$   & $L_{\reg}^{\PS}(s,\Pi,\Lambda,1)$ & $L_{\ex}^{\PS}(s,\Pi,\Lambda,1)$\\
\midrule
\nosf{I}    & $\chi_1\times\chi_2\rtimes\sigma$  &  all
                 & $L(s,\sigma)L(s,\chi_1\sigma)L(s,\chi_2\sigma)L(s,\chi_1\chi_2\sigma)$ & $1$\\
\nosf{IIa}  & $\chi \St\rtimes\sigma$      &  all  $\neq(\chi\sigma)\circ N_{K/k}$           & $L(s,\sigma)L(s,\chi^2\sigma)L(s,\nu^{1/2}\chi\sigma)$ & $1$\\
\nosf{IIb}  &$\chi\mathbf{1}\rtimes\sigma$& $(\chi\sigma)\circ N_{K/k}$ & $L(s,\sigma)L(s,\chi^2\sigma)L(s,\nu^{-1/2}\chi\sigma)$ &
$L(s,\chi\sigma\nu^{1/2})$\\
\nosf{IIIa} &$\chi\rtimes \sigma \St$     & all                & $L(s,\nu^{1/2}\chi\sigma)L(s,\nu^{1/2}\sigma)$ & $1$\\
\nosf{IIIb} &$\chi\rtimes \sigma\mathbf{1}$& none             & ---                                            & ---\\     
\nosf{IVa}  &$\sigma \St_{G}$             & all $\neq\sigma\circ N_{K/k}$  & $L(s,\nu^{3/2}\sigma)$ & $1$\\
\nosf{IVb}  &$L(\nu^2,\nu^{-1}\sigma \St)$& $\sigma\circ N_{K/k}$ & $L(s,\nu^{3/2}\sigma)L(s,\nu^{-1/2}\sigma)$ & $1$\\
\nosf{IVc}  &$L(\nu^{3/2} \St,\nu^{-3/2}\sigma)$& none & ---    & ---\\
\nosf{IVd}  &$\sigma\mathbf{1}_G$        & none           & --- & ---\\
\nosf{Va}   &$\delta([\xi,\nu\xi],\nu^{-1/2}\sigma)$&  all $\neq\sigma\circ N_{K/k},\ \xi\sigma\circ N_{K/k}$ & $L(s,\nu^{1/2}\sigma)L(s,\nu^{1/2}\xi\sigma)$ & $1$ \\
\nosf{Vb}   &$L(\nu^{1/2}\xi \St, \nu^{-1/2}\sigma)$ & $\sigma\circ N_{K/k}$ if $\xi\neq\chi_{K/k}$    & $L(s,\nu^{-1/2}\sigma)L(s,\nu^{1/2}\xi\sigma)$ & $L(s,\nu^{1/2}\sigma)$\\
\nosf{Vc}   &$L(\nu^{1/2}\xi \St, \nu^{-1/2}\xi\sigma)$ & $(\xi\sigma)\circ N_{K/k}$ if $\xi\neq\chi_{K/k}$ & $L(s,\nu^{1/2}\sigma)L(s,\nu^{-1/2}\xi\sigma)$ & $L(s,\nu^{1/2}\xi\sigma)$\\
\nosf{Vd}   &$L(\nu\xi,\xi\rtimes\nu^{-1/2}\sigma)$ & $\sigma\circ N_{K/k}$ if $\xi=\chi_{K/k}$  & $L(s,\nu^{-1/2}\sigma)L(s,\nu^{-1/2}\xi\sigma)$ & $L(s,\nu^{1/2}\sigma)L(s,\nu^{1/2}\xi\sigma)$\\
\nosf{VIa}  &$\tau(S,\nu^{-1/2}\sigma)$  & all $\neq\sigma\circ N_{K/k}$  & $L(s,\nu^{1/2}\sigma)^2$ & $1$\\
\nosf{VIb}  &$\tau(T,\nu^{-1/2}\sigma)$  & $\sigma\circ N_{K/k}$& $L(s,\nu^{1/2}\sigma)$ & $L(s,\nu^{1/2}\sigma)$\\
\nosf{VIc}  &$L(\nu^{1/2} \St,\nu^{-1/2}\sigma)$ &  none   & --- & ---\\
\nosf{VId}  &$L(\nu,1\rtimes\nu^{-1/2}\sigma)$ &   none   & --- & ---\\
\nosf{VII}  &$\chi\rtimes\pi$              & all          & $1$ & $1$\\
\nosf{VIIIa}&$\tau(S,\pi)$                 & all with $\Hom_{\widetilde{T}}(\pi,\Lambda)\neq0$& $1$ & $1$\\
\nosf{VIIIb}&$\tau(T,\pi)$                 & all with $\Hom_{\widetilde{T}}(\pi,\Lambda)=0$   & $1$ & $1$\\
\nosf{IXa}  &$\delta(\nu\xi,\nu^{-1/2}\pi(\mu))$& all, except $\mu,\mu'$ if $\xi=\chi_{K/k}$       & $1$ & $1$\\
\nosf{IXb}  & $L(\nu\xi,\nu^{-1/2}\pi(\mu))$    & $\mu,\mu'$ if $\xi=\chi_{K/k}$         & $1$ & $1$\\
\nosf{X}    &$\pi\rtimes\sigma$            & all with $\Hom_{\widetilde{T}}(\sigma\pi,\Lambda)\neq0$               & $L(s,\sigma)L(s,\omega_{\pi}\sigma)$ & $1$\\
\nosf{XIa}  &$\delta(\nu^{1/2}\pi,\nu^{-1/2}\sigma)$ &  all $\neq\sigma\circ N_{K/k}$ if $\Hom_{\widetilde{T}}(\sigma\pi,\Lambda)\neq0$  & $L(s,\nu^{1/2}\sigma)$  & $1$ \\
\nosf{XIb}  &$L(\nu^{1/2}\pi,\nu^{-1/2}\sigma)$ &   $\sigma\circ N_{K/k}$ if $\Hom_{\widetilde{T}}(\pi,\C)\neq0$       & $L(s,\nu^{-1/2}\sigma)$ & $L(s,\nu^{1/2}\sigma)$\\
            & cuspidal generic         & see \cite{Prasad_TaklooBighash}   & $1$ & $1$  \\
\nosf{Va*}  & $\theta_-(\sigma\St,\xi\sigma\St)$ & $\sigma\circ N_{K/k}$ for $\xi=\chi_{K/k}$& $1$ & $L(s,\nu^{1/2}\sigma)L(s,\nu^{1/2}\xi\sigma)$ \\
\nosf{XIa*} & $\theta_-(\sigma\St,\sigma\pi)$, & $\sigma\circ N_{K/k}$ for $\Hom_{\widetilde{T}}(\pi^\mathrm{JL},1)\neq0$  & $1$ & $L(s,\nu^{1/2}\sigma)$ \\
            & other cuspidal non-generic & see \cite{Prasad_TaklooBighash} & $1$ & $1$ \\\bottomrule
\end{tabular}
\end{center}
\end{footnotesize}
In lines \nosf{IXa} and \nosf{IXb}, $\pi(\mu)$ denotes the cuspidal irreducible representation of $\Gl(2)$ attached to a smooth character $\mu$ of $L^\times$ with Galois conjugate $\mu'\neq\mu$ where $L/k$ is the quadratic extension corresponding to $\xi$, compare \cite{Roberts-Schmidt_Bessel}.
In lines \nosf{XIa}, \nosf{XIb} and \nosf{XIa*}, $\pi$ is a cuspidal irreducible representation of $\Gl(2)$ with trivial central character.
For the last three lines, the exceptional factor has only been determined under the restriction to odd residue characteristic of $k$ \cite{Danisman_Annals}.
\end{table}
\end{landscape}

\begin{table}
\caption{Spinor $L$-factors for irreducible smooth representations $\Pi$ of $\GSp(4,k)$.\label{tab:spinor_factors}}
\begin{footnotesize}
\begin{center}
\begin{tabular}{llll}
\toprule
Type & $\Pi\in\CCC_{\GSp(4,k)}$                   & $L^{\PS}(s,\Pi,\Lambda,1)$\\
\midrule
\nosf{I} & $\chi_1\times\chi_2\rtimes\sigma$      & $L(s,\sigma)L(s,\chi_1\sigma)L(s,\chi_2\sigma)L(s,\chi_1\chi_2\sigma)$\\
\nosf{IIa}& $\chi \St\rtimes\sigma$               & $L(s,\sigma)L(s,\chi^2\sigma)L(s,\nu^{1/2}\chi\sigma)$\\
\nosf{IIb}  &$\chi\mathbf{1}\rtimes\sigma$              &$L(s,\sigma)L(s,\chi^2\sigma)L(s,\nu^{-1/2}\chi\sigma) L(s,\nu^{1/2}\chi\sigma)$\\
\nosf{IIIa} &$\chi\rtimes \sigma \St$                   & $L(s,\nu^{1/2}\chi\sigma)L(s,\nu^{1/2}\sigma)$\\
\nosf{IIIb} &$\chi\rtimes \sigma\mathbf{1}$             & $L(s,\nu^{1/2}\chi\sigma)L(s,\nu^{1/2}\sigma)L(s,\nu^{-1/2}\chi\sigma)L(s,\nu^{-1/2}\sigma)$\\
\nosf{IVa}  &$\sigma \St_{G}$                           & $L(s,\nu^{3/2}\sigma)$\\
\nosf{IVb}  &$L(\nu^2,\nu^{-1}\sigma \St)$              & $L(s,\nu^{3/2}\sigma)L(s,\nu^{-1/2}\sigma)$\\
\nosf{IVc}  &$L(\nu^{3/2} \St,\nu^{-3/2}\sigma)$        & $L(s,\nu^{3/2}\sigma)L(s,\nu^{1/2}\sigma)L(s,\nu^{-3/2}\sigma)$\\
\nosf{IVd}  &$\sigma\mathbf{1}_G$                       & no Bessel model\\
\nosf{Va}   &$\delta([\xi,\nu\xi],\nu^{-1/2}\sigma)$    & $L(s,\nu^{1/2}\sigma)L(s,\nu^{1/2}\xi\sigma)$\\
\nosf{Vb}   &$L(\nu^{1/2}\xi \St, \nu^{-1/2}\sigma)$    & $L(s,\nu^{1/2}\xi\sigma)
L(s,\nu^{1/2}\sigma)L(s,\nu^{-1/2}\sigma)$\\
\nosf{Vc}   &$L(\nu^{1/2}\xi \St, \nu^{-1/2}\xi\sigma)$ & $L(s,\nu^{1/2}\sigma)L(s,\nu^{1/2}\xi\sigma)L(s,\nu^{-1/2}\xi\sigma)$\\
\nosf{Vd}   &$L(\nu\xi,\xi\rtimes\nu^{-1/2}\sigma)$     & $L(s,\nu^{1/2}\sigma)L(s,\nu^{1/2}\xi\sigma)L(s,\nu^{-1/2}\sigma)L(s,\nu^{-1/2}\xi\sigma)$\\
\nosf{VIa}  &$\tau(S,\nu^{-1/2}\sigma)$                 & $L(s,\nu^{1/2}\sigma)^2$\\
\nosf{VIb}  &$\tau(T,\nu^{-1/2}\sigma)$                 & $L(s,\nu^{1/2}\sigma)^2$\\
\nosf{VIc}  &$L(\nu^{1/2} \St,\nu^{-1/2}\sigma)$        & $L(s,\nu^{1/2}\sigma)^2L(s,\nu^{-1/2}\sigma)$\\
\nosf{VId}  &$L(\nu,1\rtimes\nu^{-1/2}\sigma)$          & $L(s,\nu^{1/2}\sigma)^2L(s,\nu^{-1/2}\sigma)^2$\\
\nosf{VII}  &$\chi\rtimes\pi$                           & $1$\\
\nosf{VIIIa}&$\tau(S,\pi)$                              & $1$\\
\nosf{VIIIb}&$\tau(T,\pi)$                              & $1$\\
\nosf{IXa}  &$\delta(\nu\xi,\nu^{-1/2}\pi)$             & $1$\\
\nosf{IXb}  & $L(\nu\xi,\nu^{-1/2}\pi)$                 & $1$\\
\nosf{X}    &$\pi\rtimes\sigma$                         & $L(s,\sigma)L(s,\omega_{\pi}\sigma)$\\
\nosf{XIa}  &$\delta(\nu^{1/2}\pi,\nu^{-1/2}\sigma)$    & $L(s,\nu^{1/2}\sigma)$\\
\nosf{XIb}  &$L(\nu^{1/2}\pi, \nu^{-1/2}\sigma)$        & $L(s,\nu^{1/2}\sigma)L(s,\nu^{-1/2}\sigma)$\\
            & \ cuspidal generic                        & $1$                   \\
\nosf{Va*}  & $\theta_-(\sigma\St,\xi\sigma\St)$      & $L(s,\nu^{1/2}\sigma)L(s,\nu^{1/2}\xi\sigma)$\\
\nosf{XIa*} & $\theta_-(\sigma\St,\sigma\pi)$         & $L(s,\nu^{1/2}\sigma)$\\
            & other cuspidal non-generic              & $1$\\
            \bottomrule
\end{tabular}
\end{center}
\end{footnotesize}
\begin{small}
This table lists Piatetski-Shapiro's spinor $L$-factors attached to infinite-dimensional irreducible smooth representations $\Pi$ of $\GSp(4,k)$ and every Bessel model $(\Lambda,\psi)$ of $\Pi$, split or anisotropic.
For the last three lines, this has only been shown
under the restriction to odd residue characteristic of $k$.
\end{small}
\end{table}

\end{document}